\documentclass{amsart}
\usepackage{amssymb,amsmath}
\usepackage{amsfonts}
\usepackage{pstricks,pst-node}
\usepackage[all]{xy}
\numberwithin{equation}{section}
\theoremstyle{plain}
\newtheorem{lemma}{Lemma}

\newtheorem{corollary}{Corollary}
\begin{document}
\title{Standard and Weierstrass spin groups on hyperelliptic Riemann surfaces}
\author{K.~M.Bugajska}
\address{Department of Mathematics and Statistics,
York Yniversity,
Toronto, ON, M3J 1P3,
Canada}
\begin{abstract}
We show that at any standard or Weierstrqass point $P$ on a hyperelliptic Riemann surface $\Sigma$ equipped with a nonsingular even spin structure, we may attach a spin group $\mathtt{G}_{P}$ in a natural way. All such spin groups are isomorphic to each other and to the alternating group $\mathtt{A}_{g}\triangleleft{\mathtt{S}_{g}}$. Moreover, for any two vertices of the same spin graph we construct a natural, unique  isomorphism between the corresponding, mutually conjugate, spin groups.  
\end{abstract}

\maketitle

\section{Introduction}
 It is well known ~\cite{RM95},~\cite{DV11},~\cite{RG67},  that when $\xi_{\epsilon}$ is an even nonsingular spin bundle on a Riemann surface $\Sigma$ then  each point $P\in{\Sigma}$ determines the unique section of the line bundle $\xi_{\epsilon}$ with a single, simple pole at $P$. The zero divisor $\mathcal{A}^{\epsilon}_{P}$ of this section is an integral divisor of degree $g$ whose index of specialty is zero. The points of $\Sigma$ which form the divisor $\mathcal{A}^{\epsilon}_{P}$ are exactly those points whose image under the Jacobi mapping $\Phi_{P}$ originating at $P$ lies in the theta divisor $\Theta_{\epsilon}$   ~\cite{FK92},~\cite{FK01},~\cite{RG67}.
 
It appears  that for any $P\in{\Sigma}$ the spin structure $\xi_{\epsilon}$ allows us to associate a set of points $\{\mathcal{S}^{\epsilon}_{P}\}\subset{\Sigma}$  (see ~\cite{KB13}). In other words, $\xi_{\epsilon}$ introduces a foliation of $\Sigma$ which we call  the $\epsilon$-foliation.  Each leaf  of this foliation carries an additional structure which determines an appropriate graph (called a spin-graph).   For any point $P$ of a leaf we
    will denote such graph by $\mathcal{S}^{\epsilon}_{P}$. Obviously, we have ${\mathcal{S}^{\epsilon}_{P}}={\mathcal{S}^{\epsilon}_{Q}}$ for any two points $P$ and $Q$ of te same leaf.  Besides,  each leaf corresponds to a well defined subset of $\Theta_{\epsilon}$ (more precisely, to the smallest subset of $\Theta_{\epsilon}$ containing the set $\Phi_{P}(\{\mathcal{A}^{\epsilon}_{P}\})$ for
      any point $P$ of the leaf, see definition $1$). 

When $\Sigma$ is a hyperelliptic Riemann surface with the genus $g$, then all leaves have finite number of points. Almost all leaves consist of $2g+2$ points and these points form vertices of a standard spin-graph of genus $g$. Such graph is totally symmetric with respect to all of its vertices.

A leaf of an $\epsilon$-foliation through a Weierstrass point has $g+1$ points. All of them are also Weierstrass points and they form vertices of (totally symmetric) Weierstrass spin-graph on $\Sigma$.

However, on any surface $\Sigma$ our $\epsilon$-foliation must have exceptional leaves. Their points will be called exceptional points and there are at most $4g$ of them. The notion of spin-graphs and the classification of hyperelliptic Riemann surfaces using exceptional spin graphs is given in [2].

In this paper we will demonstrate how a spin structure $\xi_{\epsilon}$ allows us to associate to any standard or to any  Weierstrass point $P$ a concrete group $\mathtt{G}_{P}$ which acts as a group of permutations on the set ${\{\mathcal{A}^{\epsilon}_{P}\}}\cong{\widehat{P}}=\Phi_{P}(\Sigma)\cap{\Theta_{\epsilon}}\subset{Jac\Sigma}$. We will show that any standard spin group (i.e.  associated to a standard point) and any Weierstrass spin group (i.e.  associated to a Weierstrass point) is isomorphic to the alternating group $\mathtt{A}_{g}\triangleleft{\mathtt{S}_{g}}$.  Moreover, we will  construct a concrete isomorphism $\mathcal{A}d(C):{\mathtt{G}_{P}}\rightarrow{\mathtt{G}_{Q}}$ between two spin groups only when $P$ and $Q$ are vertices of the same spin-graph $\mathcal{S}^{\epsilon}_{P}={\mathcal{S}^{\epsilon}_{Q}}$ (i.e.  when both points belong to exactly the same leaf of $\epsilon$-foliation on $\Sigma$).

 A case of a standard spin group on a surface of genus $3$ shall be investigate very thoroughly since this will provide tools to find the standard groups for the arbitrary genus $g>3$. Analogously with the Weierstrass spin groups.

The next section contais preliminaries. In section $3$ we consider a case of standard graphs on a surface of genus $3$.  We define isomorphisms between the sets $\widehat{P}$ and $\widehat{Q}$ when $PQ$ is an edge of this spin graph and show that for any face $F$ with this edge
the permutations representing this isomorphism have exactly the same parity. 

 This implies that  any spin chain $\mathsf{W}_{P}$ at $P$ produces an even permutation of the set $\widehat{P}\cong{\{\mathcal{A}^{\epsilon}_{P}\}}$. Since some chains produce generators of the alternating group we obtain that $\mathtt{G}_{P}\cong{\mathtt{A}_{3}}$.

In section $4$ we consider an arbitrary genus $g\geq{3}$. When the genus is $g>3$ then we use $3$-cells which were introduced in section $2$. Since any face $F\subset{\mathcal{S}^{\epsilon}_{P}}$ determines a unique $3$-cell in $\mathcal{S}^{\epsilon}_{P}$, it generates unique isomorphisms between the appropriate $g$ element sets associated to the vertices of $F$ resprctively. We will show that the parities of permutations representing these isomorphisms depend only on concrete vertices of an edge of $F$  and not on the  face $F$ itself.  As a consequence of this fact, any spin chain $\mathsf{W}_{P}$ at point $P$ produses an even permutation of the set $\widehat{P}\cong{\{\mathcal{A}^{\epsilon}_{P}\}}$.  Since we have spin chains that produce generators of the alternatig group we obtain immediately  that $\mathtt{G}_{P}\cong{\mathtt{A}_{g}}\triangleleft{\mathtt{S}_{g}}$.

In section $5$, for any two vertices $P$ and $Q$ of a standard spin graph of any genus,  we will construct the shortest path $C_{PQ}$ from $P$ to $Q$ in $\mathcal{S}^{\epsilon}_{P}={\mathcal{S}^{\epsilon}_{Q}}$.  This will allow us to find a concrete bijection between the set of spin chains at point $P$ and the set of chains at point $Q$.  Then we find a unique permutation $\sigma^{C}_{PQ}$ which represents an isomorphism from the set $\widehat{P}$ to $\widehat{Q}$. This will  produce an isomorphism $\mathtt{G}_{Q}={\mathcal{A}d(\sigma^{C}_{PQ})}{\mathtt{G}_{P}}$.

In section $6$ we show that the spin group at any Weierstrase point of $\Sigma$ is also  isomorphic to the alternating group $\mathtt{A}_{g}$.

Thus, at almost any point $P$ of a surface $\Sigma$ any even nonsingular spin bundle $\xi_{\epsilon}$ associates the spin group $\mathtt{G}_{P}$ which is isomorphic to the alternating group $\mathtt{A}_{g}$.
   There is only a finite, smaller than or equal to $4g$ number of points  on $\Sigma$ for which spin groups have to be different.  For these exceptional points even the spin groups at vertices of the same (exceptional) spin graph may be distinct.  All such exceptional spin groups as well as the classification of hyperelliptic Riemann surfaces based on exceptional groups is given in [3].

\section{Preliminaries}
Let $P$ be a point of a hyperelliptic Riemann surface $\Sigma$ equipped with a non singular even spin structure $\xi_{\epsilon}$. The divisor of the  unique section $\sigma_{P}$ of the  line bundle $\xi_{\epsilon}$  which has  a single simple pole at $P$  is $(\sigma_{P})=div\sigma_{P}=P^{-1}\mathcal{A}^{\epsilon}_{P}$, ~\cite{RG67}. Here $\mathcal{A}^{\epsilon}_{P}$  is an integral divisor of degree $g$ with vanishing index of specialty. Let $\Phi_{P}:\Sigma\rightarrow{Jac\Sigma}$ denote the Jacobi mapping originated at $P$.  We introduce te set $\widehat{P}_{\epsilon}$ as follows:
\begin{equation}
\widehat{P}_{\epsilon}:=\{\Phi_{P}(P');P'\in\{\mathcal{A}^{\epsilon}_{P}\}\}=\Phi_{P}(\Sigma)\cap{\Theta}_{\epsilon}
\end{equation} 
where $\Theta_{\epsilon}\subset{Jac\Sigma}$ is the appropriate theta-divisor.

Let $\mathcal{S}^{\epsilon}_{P}$ denote the spin graph through $P$ on $\Sigma$ (as introduced in ~\cite{KB13}) and let $\{\mathcal{S}^{\epsilon}_{P}\}$ denote the set of its vertices. The set of vertices of any spin graph may be equivalently given as follows
\newtheorem{definition}{Definition}
\begin{definition}
For any point $P\in\Sigma$ the set $\{\mathcal{S}^{\epsilon}_{P}\}\subset{\Sigma}$ is the smallest set of points containing $P$ with the property that for each $P'\in\{\mathcal{S}^{\epsilon}_{P}\}$ all points whose image under $\Phi_{P'}$ lie in the theta divisor $\Theta_{\epsilon}$ are also in $\{\mathcal{S}^{\epsilon}_{P}\}$ 
\end{definition}
When $P$ is a standard point then the spin graph $\mathcal{S}^{\epsilon}_{P}$ through this point is called a standard graph. It  has $2g+2$ vertices which will be enumarated as $\{\mathcal{S}^{\epsilon}_{P}\}={\{P,P_{k},\widetilde{P},\widetilde{P}_{k}; k=1,\ldots{g}\}}$ and  no pair of conjugate vertices is connected by an edge. Moreover, for each vertex $Q\in{\{\mathcal{S}^{\epsilon}_{P}\}}$, the enumeration of the points $P_{k}$, $k=1,\ldots,g$ of $\mathcal{A}^{\epsilon}_{P}$  together with the equivalences of the divisors $(\sigma_{P})\cong{(\sigma_{P_{k}})}\cong(\sigma_{\widetilde{P_{k}}})\cong(\sigma_{\widetilde{P}})$ produces a natural ordering $Q_{k}$, $k=1,2,..,g$ of the points of the divisor $\mathcal{A}^{\epsilon}_{Q}$ (and hence of the set $\widehat{Q_{\epsilon}}\subset{\Theta_{\epsilon}}$). 

 The set of all edges in a standard graph $\mathcal{S}^{\epsilon}_{P}$ is in one-one correspondence with an appropriate subset of the theta divisor. More precisely, we have the following:
\begin{lemma}
For any standard point $P$ of a surface $\Sigma$ there is one-one correspondence between the set 
\begin{equation*}
\Upsilon^{\epsilon}_{P}:={\bigcup_{P'\in\{\mathcal{S}_{P}^{\epsilon}\}}{\widehat{P'}_{\epsilon}}\subset\Theta_{\epsilon}}
\end{equation*}
and the set of all edges of the spin graph $\mathcal{S}^{\epsilon}_{P}$
\end{lemma}
\begin{proof}
Using the equivalence of appropriate divisors (i.e.the fact that $\frac{P_{i}}{P}\cong{\frac{\widetilde{P}}{\widetilde{P_{i}}}}$ etc.) we see that the cardinality of the set $\Upsilon^{\epsilon}_{P}$ given by $2\sum^{g}_{k=1}k$ is equal to    the number of edges of $\mathcal{S}^{\epsilon}_{P}$ i.e. to $\frac{1}{2}|\{\mathcal{S}^{\epsilon}_{P}\}|g$. The one-to-one correspondence is given by the following: an edge $P_{i}\widetilde{P_{k}}$ corresponds to $\Phi_{P_{i}}(\widetilde{P_{k}})$ when $i<k$ and to $\Phi_{\widetilde{P}_{k}}(P_{i})$ when $i>k$. An edge $PP_{i}$ corresponds to the point $\Phi_{P}(P_{i})$ in $\Theta_{\epsilon}$ whereas an edge $\widetilde{P}\widetilde{P_{i}}$ corresponds to $\Phi_{\widetilde{P}}(\widetilde{P_{i}})$; $i,k=1,\ldots,g$.
\end{proof}
When the genus $g>2$ than all faces of a standard spin graph are quadrangles and no two vertices of a face are mutually conjugate.  We have

\begin{lemma}\hfill
Let $\mathcal{F}$ denote the set of all distinct quadrangle faces of a standard graph $\mathcal{S}^{\epsilon}_{P}$ on a surface of genus $g>2$.  Let $\mathcal{F}_{P'}\subset{\mathcal{F}}$  denote the set of all faces with  a common vertex, say $P'$, of the graph $\mathcal{S}^{\epsilon}_{P}$. Then
\begin{enumerate}
\item For each vertex $P'$ the cardinality of the set $\mathcal{F}_{P'}$ is qual to 
\begin{equation*}
|\mathcal{F}_{P'}|={\frac{1}{2}g(g-1)(g-2)}
\end{equation*}
\item $|\mathcal{F}|={\frac{1}{4}(g+1)g(g-1)(g-2)}$
\end{enumerate}
\end{lemma}

\begin{proof}
For the first part it is enough to consider a case when $P'=P$. Now, each quadrangle face that belongs to $\mathcal{F}_{P}$ has vertises $\{P,P_{k},P_{l},\widetilde{P}_{m}\}$ with $k,l,m=1,2,...g$ all distinct from each other and  hence the property $(1)$ follows. To show the second property, let $\mathcal{F}_{1}$ denote the set of all faces with common vertex $P_{1}$ which are not in the set $\mathcal{F}_{P}\cap{\mathcal{F}_{\widetilde{P}}}$. Each face in $\mathcal{F}_{1}$ has vertices $\{P_{1},P_{k},\widetilde{P_{l}},\widetilde{P_{m}}\}$ with $k,l,m\geq{2}$ and all distinct. Let $\mathcal{F}_{2}$ be the set of  quadrangle faces with common vertex $P_{2}$ whose intersection with the set $\mathcal{F}_{P}\cap{\mathcal{F}_{\widetilde{P}}}\cap{\mathcal{F}_{1}}$ is empty and so on. Since $\mathcal{F}={\mathcal{F}_{P}\cup{\mathcal{F}_{\widetilde{P}}}\cup{}\mathcal{F}_{1}}\cup\ldots\cup\mathcal{F}_{g-1}$ the formula $(2)$ is true. 
\end{proof}
Although we will often denote a face $F\subset{\mathcal{S}^{\epsilon}_{P}}$ as  $F=QSRP$ , this notation does not indicate any concrete ordering of the vertices $Q$, $R$, $S$, $P$  which determine this face $F$.  On the contrary to this, when we  write a loop $\mathsf{L}=QSR..PQ$ in $\mathcal{S}_{P}$ at a vertex $Q$, or a path $C=QS..PR$  from a vertex $Q$ to $R$ in $\mathcal{S}_{P}$, then the order of writing of vertices does indicate the direction of traveling of a loop or of a path respectively. 

When the genus $g$ of a surface $\Sigma$ is $g>3$ then for any standard  spin graph $\mathcal{S}^{\epsilon}_{P}$ we may naturally define $n$-cells $\Gamma^{n}\subset{\mathcal{S}^{\epsilon}_{P}}$ as follows
\begin{definition}
For any $n=3,4,\ldots,g-1$ an $n$-cell of  $\mathcal{S}^{\epsilon}_{P}$ is obtained by deleting the vertices $\{R_{i_{1}},\ldots,R_{i_{g-n}},\widetilde{R_{i_{1}}},\ldots,\widetilde{R_{i_{g-n}}}\}$ from the set $\{\mathcal{S}^{\epsilon}_{P}\}$ of all vertices  and by deleting each edge of the graph which has $R_{i_{k}}$ or $\widetilde{R_{i_{k}}}$, (for $k=1,\ldots,g-n$) as its vertex.
\end{definition}
\newtheorem{remark}{Remark}
\begin{remark}
Any $n$-cell of a standard spin graph $\mathcal{S}^{\epsilon}_{P}$ belongs to the isomorphic class $\mathcal{S}(n)$ of standard spin graphs on a hyperelliptic Riemann surface of genus equal to $n$.
\end{remark}
When we rename the vertices $\{P,P_{k},\widetilde{P},\widetilde{P_{k}}; k=1,\ldots,g\}$ of a standard spin graph $\mathcal{S}^{\epsilon}_{P}$ as follows:  $R_{i}=P_{i}$ for $i=1,\ldots,g$ and  $R_{g+1}=P$, then we  immediately see that:
\begin{remark}
The number of distinct $n$-cells in a standard spin graph $\mathcal{S}^{\epsilon}_{P}$ is equal to the number of all combinations of $g+1$ objects taken $g-n$ at a time. In particular, there are $g+1$ distinct $(g-1)$-cells $\Gamma^{g-1}_{k}\subset{\mathcal{S}^{\epsilon}_{P}}$,  $k=1,2,...,g+1$, whose set of vertices is obtained from the set  $\{\mathcal{S}^{\epsilon}_{P}\}$  by deleting the vertices $R_{k}$ and $\widetilde{R}_{k}$ resprctively.
\end{remark}

We will use the following notation: If an object $\mathcal{O}$ is not a set then by $\{\mathcal{O}\}$   we will denote  the set of its points. We will often denote a face $F\subset{\mathcal{S}^{\epsilon}_{P}}$  determined by vertices $\{F\}={\{Q,S,R,P\}}$ as $F={QSRP}$ and this latter notation does not indicate any concrete ordering of the vertices of $F$.

\section{Standard Spin Group $\mathtt{G}(3)$}
Let $\Sigma$ be a surface of genus $g=3$ and let $P$ be a standard point of $\Sigma$ with respect to some nonsingular even (fixed) spin structure $\xi_{\epsilon}$   on this surface . Using our previous enumeration of the set of vertices of $\mathcal{S}^{\epsilon}_{P}$ we will introduce the following notations:
\begin{equation*}
\widehat{P_{\epsilon}}=\Phi_{P}(\Sigma)\cap{\Theta_{\epsilon}}\cong{\{\frac{P_{1}}{P},\frac{P_{2}}{P},\frac{P_{3}}{P}\}}\Leftrightarrow{\{1,2,3\}_{P}}
\end{equation*}
 Similarly, for any other vertex $Q$ of $\mathcal{S}^{\epsilon}_{P}$ we will identify the ordered  set $\widehat{Q_{\epsilon}}={\Phi_{Q}(\Sigma)}\cap{\Theta_{\epsilon}}$ with the set ${\{1,2,3\}}_{Q}$. We recall that any concrete enumeration of the points $P_{1},P_{2}$ and $P_{3}$ of the integral divisor $\mathcal{A}^{\epsilon}_{P}$ determines a unique enumeration of the poins of the divisor  $\mathcal{A}^{\epsilon}_{Q}$ (and hence the enumeration of the points of $\widehat{Q_{\epsilon}}$).  For example, for the vertex $P_{2}$ we have $\widehat{P_{2}}\cong\{\frac{P}{P_{2}},\frac{\widetilde{P_{1}}}{P_{2}},\frac{\widetilde{P_{3}}}{P_{2}}\}$ so we may naturally identify this set   with  $\{1,2,3\}_{P_{2}}$.
  
 The standard spin graph $\mathcal{S}^{\epsilon}_{P}\in{\mathcal{S}(3)}$ is given by the Pict1. To draw this graph we will use a convention that the indices of the points $P_{i}$, $i=1,2,3$ , are increasing from the bottom up. Of course,  our all  results do not depend neither from the enumeration of the points of $\mathcal{A}^{\epsilon}_{P}$ nor from the fact whether the indices  are increasing or decreasing from the bottom up.

 \begin{pspicture}(-3,-3)(4,3)
 \psline[showpoints=true]%
 (-1,-1)(1,-1)(2,0)
 \psline[showpoints=true]%
 (-1,1)(1,1)(2,2)
 \psline[showpoints=true, linestyle=dashed]%
 (0,2)(0,0)
 \psline(-1,-1)(-1,1)
 \psline(1,-1)(1,1)
 \psline(-1,1)(0,2)
 \psline(0,2)(2,2)
 \psline[linestyle=dashed]%
 (-1,-1)(0,0)
 \psline[linestyle=dashed]%
 (0,0)(2,0)
 \psline (2,0)(2,2)
 
 \rput(-0.2,-1.5){\rnode{A}{$\mathcal{S}(3)$}}
 \rput(0,-2.1){\rnode{B}{Pict.1}}
 
 \rput(-1.4,-0.8){\rnode{a}{$P^{\mathbf{R}}$}}
 \rput(-0.4,0.1){\rnode{b}{$P_{2}^{\mathbf{L}}$}}
 \rput(-1.4,1.2){\rnode{c}{$P_{3}^{\mathbf{R}}$}}
 \rput(0.4,2.4){\rnode{d}{${\widetilde{P_{1}}}^{\mathbf{L}}$}}
 \rput(2.4,2.3){\rnode{e}{${\widetilde{P}}^{\mathbf{L}}$}}
 \rput(0.8,1.3){\rnode{f}{${\widetilde{P_{2}}}^{\mathbf{R}}$}}
 \rput(1.5,-1){\rnode{g}{${P_{1}}^{\mathbf{R}}$}}
 \rput(2.5,0){\rnode{h}{${\widetilde{P_{3}}}^{\mathbf{L}}$}}
 \end{pspicture}

 Let $Q$ be any vertex of the graph $\mathcal{S}^{\epsilon}_{P}$ given by Pict1. Let $\widehat{Q}\cong{\{\frac{Q_{1}}{Q},\frac{Q_{2}}{Q},\frac{Q_{3}}{Q}\}}\Leftrightarrow{\{1,2,3,\}}_{Q}$ be the ordering of elements of $\widehat{Q}$ uniquely determined by an enumeration of elements of $\widehat{P}$. From now on, the index $\epsilon$, which indicates the characteristic $[\epsilon]$ of a fixed spin structure on $\Sigma$, will be omitted. Let $\{\overrightarrow{QQ_{1}},\overrightarrow{QQ_{2}},\overrightarrow{QQ_{3}}\}$ be a triple of vectors ordered by the enumeration of $\widehat{Q}$, i.e. a triple of vectors with a concrete orientation.  From the Pict1 we observe  that the triple $\{\overrightarrow{PP_{1}},\overrightarrow{PP_{2}},\overrightarrow{PP_{3}}\}$ has the right hand orientation whereas the triple $\{\overrightarrow{P_{2}P},\overrightarrow{P_{2}\widetilde{P_{1}}},\overrightarrow{P_{2}\widetilde{P_{3}}}\}$ associated to $\widehat{P_{2}}\cong{\{1,2,3\}}_{P_{2}}$ has the left hand orientation. 
 
 Now, the natural enumerations (determined by the enumeration of the elements of $\mathcal{A}_{P}$) of the sets $\widehat{P}$, $\widehat{P_{1}}$, $\widehat{\widetilde{P_{2}}}$ and $\widehat{P_{3}}$ produce the right hand oriented triples of vectors, whereas the natural ordering of the sets $\widehat{\widetilde{P}}$, $\widehat{\widetilde{P_{1}}}$, $\widehat{P_{2}}$ and $\widehat{\widetilde{P_{3}}}$  produce the left oriented triples of vectors. The right $\mathfrak{R}$ and the left $\mathfrak{L}$ hand of orientations of the appropriate triples of vectors associated to each vertex of $\mathcal{S}_{P}$ are indicated on the Pict1 respectively.
 
 Let $F$ be a (quadrilateral) face of the graph $\mathcal{S}_{P}\in{\mathcal{S}(3)}$ with vertices ,say, $\{F\}={\{Q_{1},Q_{2},Q_{3},Q_{4}\}}$. Any oriented edge $\overrightarrow{Q_{i}Q_{j}}$ of $F$ corresponds to the "`flip"' from the Jacobi mapping $\Phi_{Q_{i}}$ to $\Phi_{Q_{j}}$. We will call this a Jacobi rotation and we will denote this by $\Phi^{F}_{Q_{i}Q_{j}}$.

 When face $F$ and  its  edge $Q_{i}Q_{j}$ is fixed, see Pict.$2$, then $\Phi_{Q_{i}Q_{j}}$ determines unique isomorphism 
 \begin{equation*}
 \Phi^{F}_{Q_{i}Q_{j}}\equiv{\widehat{F}}:\widehat{Q_{i}}\cong{\{1,2,3\}}_{Q_{i}}\rightarrow{\widehat{Q_{j}}}\cong{\{1,2,3,\}}_{Q_{j}}
 \end{equation*}
 in a natural way. Thus, for $i=1$ and for $j=2$  we read from Pict2 that $\Phi^{F}_{Q_{1}Q_{2}}$ maps ${\frac{Q_{4}}{Q_{1}}}\rightarrow{\frac{Q_{1}}{Q_{2}}}$, ${\frac{Q_{2}}{Q_{1}}}\rightarrow{\frac{Q_{3}}{Q_{2}}}$ and the remaining point of $\widehat{Q_{1}}$ into the remaining point of the set $\widehat{Q_{2}}$.  Similarly for ${\Phi}^{F}_{Q_{2}Q_{1}}$.  In this way, we obtain unique bijections (denoted by $\widehat{F}$) between the sets 
 \begin{equation}
 \widehat{F}: \widehat{Q_{1}}\leftrightarrow{\widehat{Q_{2}}}\leftrightarrow{\widehat{Q_{3}}}\leftrightarrow{\widehat{Q_{4}}}
 \end{equation}
Let $\sigma^{F}_{Q_{i}Q_{j}}:\{1,2,3,\}_{Q_{i}}\rightarrow\{1,2,3\}_{Q_{j}}$ be the permutation
\begin{equation}
 \begin{pmatrix}1&2&3\\\sigma^{F}_{Q_{i}Q_{j}}(1)&\sigma^{F}_{Q_{i}Q_{j}}(2)&\sigma^{F}_{Q_{i}Q_{j}}(3)\end{pmatrix}
 \end{equation}
  corresponding to the isomorphism $\widehat{F}:\widehat{Q_{i}}\rightarrow{\widehat{Q_{j}}}$, ($i,j=1,2,3,4$, $i\neq{j}$). Obviously, the permutations $\sigma^{F}_{Q_{i}Q_{j}}$ and $\sigma^{F}_{Q_{j}Q_{i}}: \{1,2,3\}_{Q_{j}}\rightarrow{\{1,2,3\}_{Q_{i}}}$are inverse to each other. Moreover, for any $i,j,k\in{\{1,2,3,4\}}$ we have the property
\begin{equation}
\sigma^{F}_{Q_{i}Q_{j}}\circ{\sigma^{F}_{Q_{k}Q_{i}}}=\sigma^{F}_{Q_{k}Q_{j}}
\end{equation}

\begin{pspicture}(-2,-2)(4,2.5)
\pspolygon[showpoints=true]%
(0,0)(1,0)(2,1)(1,1)
\rput(0.7,1.2){\rnode{a}{$Q_{4}$}}
\rput(2.3,1.2){\rnode{b}{$Q_{3}$}}
\rput(1.4,0){\rnode{c}{$Q_{2}$}}
\rput(-0.3,0.2){\rnode{d}{$Q_{1}$}}

\rput(0.9,0.5){\rnode{A}{$F$}}
\rput(0.5,-1){\rnode{B}{Pict.2}}
\end{pspicture}

In particular, when we start with any vertex $Q_{i}$ of a face $F$ and move (in any direction) along the whole perimeter of $F$ then the composition of four permutations (corresponding to each oriented edge respectively) produces the identity $\mathbb{I}$. for example, let $F=PP_{1}\widetilde{P_{3}}P_{2}\subset{\mathcal{S}_{P}}$ (see Pict1).In this case we have:
\begin{equation*}
\sigma^{F}_{PP_{1}}=(132), \quad \sigma^{F}_{P_{1}\widetilde{P_{3}}}=(12), \quad \sigma^{F}_{\widetilde{P_{3}}P_{2}}=(123),\quad \sigma^{F}_{P_{2}P}=(23)
\end{equation*}
Hence $\sigma^{F}_{P_{2}P}\circ{\sigma^{F}_{\widetilde{P_{3}P_{2}}}}\circ{\sigma^{F}_{P_{1}\widetilde{P_{3}}}}\circ{\sigma^{F}_{PP_{1}}}={\mathbb{I}}={\sigma^{F}_{PP}}$ as expected.

We observe  that any edge of the standard spin graph $\mathcal{S}(3)$ is a common edge for exactly two distinct faces of the graph.

\begin{lemma}
Let $F_{1}$ and $F_{2}$ be two faces of the graph $\mathcal{S}^{\epsilon}_{P}\cong{\mathcal{S}(3)}$ with common edge $QR$,   ${F_{1}}\cap{F_{2}}=QR$.  The permutations $\sigma^{F_{1}}_{QR}:{\{1,2,3\}}_{Q}\rightarrow{\{1,2,3\}}_{R}$ and $\sigma^{F_{2}}_{QR}:{\{1,2,3\}_{Q}}\rightarrow{\{1,2,3\}_{R}}$ have exactly the same parity.
\end{lemma}
\begin{proof}\hfill
 The proof will be by construction. We observe that for any face $F$ with an edge $QR$ the parity of $\sigma^{F}_{QR}$  depends only on the orientations $\mathfrak{R}$ or $\mathfrak{L}$ of the triples of vectors associated to $\widehat{Q}$ and $\widehat{R}$ respectively and not on a face $F$ itself. More precisely we have the following:
\begin{itemize}
\item the parity of $\sigma^{F}_{QR}$ is even when the triples of vectors associated to $\widehat{Q}$ and $\widehat{R}$ respectively have the same orientations i.e. $\mathfrak{R}\mathfrak{R}$ or $\mathfrak{L}\mathfrak{L}$.
\item the parity of $\sigma^{F}_{QR}$ is odd when the corresponding orientations are opposite i.e. $\mathfrak{R}\mathfrak{L}$ or $\mathfrak{L}\mathfrak{R}$ respectively. 
\end{itemize}
Let us consider the following example: The triples of vectors $\{\overrightarrow{PP_{1}},\overrightarrow{PP_{2}},\overrightarrow{PP_{3}}\}$ and $\{\overrightarrow{P_{1}P},\overrightarrow{P_{1}\widetilde{P_{2}}},\overrightarrow{P_{}1}\widetilde{P_{3}}\}$ (corresponding to $\widehat{P}\cong{\{\frac{P_{1}}{P},\frac{P_{2}}{P},\frac{P_{3}}{P}\}}\Leftrightarrow{\{1,2,3,\}_{P}}$ and to $\widehat{P_{1}}\cong{\{\frac{P}{P_{1}},\frac{\widetilde{P_{2}}}{P_{1}},\frac{\widetilde{P_{3}}}{P_{1}}\}}\Leftrightarrow{\{1,2,3\}_{P_{1}}}$  respectively)  are right hand triples whereas the triple $\{\overrightarrow{\widetilde{P_{3}}\widetilde{P}},\overrightarrow{\widetilde{P_{3}}P_{1}},\overrightarrow{\widetilde{P_{3}}P_{2}}\}$ (associated to $\widehat{P_{3}}\cong{\{\frac{\widetilde{P}}{\widetilde{P_{3}}},\frac{P_{1}}{\widetilde{P_{3}}},\frac{P_{2}}{\widetilde{P_{3}}}\}}\Leftrightarrow{\{1,2,3\}_{\widetilde{P_{3}}}}$) is left handed. Let $F_{1}=PP_{1}\widetilde{P_{3}}P_{2}$ and let $F_{2}=PP_{1}\widetilde{P_{2}}P_{3}$ be two faces with common edge $PP_{1}$. The corresponding permutations are $\sigma^{F_{1}}_{PP_{1}}=(132)$ and $\sigma^{F_{2}}_{PP_{1}}=(123)$ and they both are even. On the other hand, the edge $P_{1}\widetilde{P_{3}}$ is common for the face $F_{1}$ and $F_{3}=P_{1}\widetilde{P_{3}}\widetilde{P}\widetilde{P_{2}}$. Now we have $\sigma^{F_{1}}_{P_{1}\widetilde{P_{3}}}=(12)$   and $\sigma^{F_{3}}_{P_{1}P_{3}}=(13)$ i.e. both permutations are odd as expected.
\end{proof} 
\begin{definition}
Any permutation $\sigma^{F}_{QR}: \{1,2,3\}_{Q}\rightarrow{\{1,2,3\}_{R}}$ will be called a spin permutation or a permutation allowed by the spin-graph $\mathcal{S}^{\epsilon}_{P}$.
\end{definition}
\begin{definition}
Let $Q$ be any vertex of $\mathcal{S}_{P}$. A spin chain $\mathsf{W}_{Q}$ at $Q$ consists of a loop ${\mathsf{L}_{Q}}\subset{\mathcal{S}_{P}}$ at $Q$  together with choices of concrete faces along each edge of this loop.
\end{definition}
We will write
\begin{equation*}
\mathsf{W}_{Q}=\{\mathsf{L}_{Q}=QQ_{1}Q_{2}\ldots{Q_{N}Q}; F_{1},F_{2},\ldots,F_{N},F_{N+1}\}
\end{equation*}
where $QQ_{1}\subset{F_{1}}$, $Q_{i-1}Q_{i}\subset{F_{i}}$ for $i=2,\ldots,N$ and $Q_{N}Q\subset{F_{N+1}}$.

This definition will be also valid for any standard spin graph on a Riemann surface with the genus $g\geq{3}$

\begin{lemma}
For any loop $\mathsf{L}_{Q}\subset{\mathcal{S}_{P}}$ at $Q$ and for any choice of faces $F_{1},\ldots,F_{N+1}$ along its edges the permutation
\begin{equation}
\sigma^{\mathsf{W}_{Q}}:=\sigma^{N+1}_{Q_{N}Q}\circ{\sigma^{N}_{Q_{N-1}Q_{N}}}\circ{\ldots}\circ{\sigma^{2}_{Q_{1}Q_{2}}}\circ{\sigma^{1}_{QQ_{1}}}
\end{equation}
which maps $\{1,2,3\}_{Q}\rightarrow{\{1,2,3\}_{Q}}$ is an even permutation.( Here $\sigma^{l}_{Q_{l-1}Q_{l}}$ denotes $\sigma^{F_{l}}_{Q_{l-1}Q_{l}}$, etc.)
\end{lemma}
\begin{proof}
Let $F$ be any face of $\mathcal{S}^{\epsilon}_{P}\cong{\mathcal{S}(3)}$ with opposite edges, say, $QR$ and $Q'R'$. We notice (see Pict1) that the orientations of the triples of vectors associated to $\widehat{Q}$ and to $\widehat{R}$ are the same  (i.e. eiter $\mathfrak{R}\mathfrak{R}$ or $\mathfrak{L}\mathfrak{L}$), or they are different (i.e. either $\mathfrak{R}\mathfrak{L}$ or $\mathfrak{L}\mathfrak{R}$) if and only if the parities of the triples of vectors associated to $\widehat{Q'}$ and to $\widehat{R'}$ are the same or they are different respectively. This implies (see lemma 3) that 
\begin{equation*}
\text{parity}(\sigma^{F}_{QR})=\text{parity}(\sigma^{F}_{Q'R'}) 
\end{equation*}
Hence we obtain that for any oriented loop $\mathsf{L}_{Q}=QQ_{1}\ldots{Q_{N}Q}$ in $\mathcal{S}_{P}$ and for any choice of faces along its edges the number of odd permutations in $\sigma^{\mathsf{W}_{Q}}$ must be even. In other words, any spin-chain $\mathsf{W}_{Q}$ at a vertex $Q$ produces an even permutation of the set $\widehat{Q}$ (or, equivalently, of the set $\{\mathcal{A}_{Q}\}$).
\end{proof}
\begin{corollary}
For any vertex $Q$ of a standard spin graph $\mathcal{S}^{\epsilon}_{P}\cong{\mathcal{S}(3)}$ the set of all permutations of $\widehat{Q}$ produced by all choices of spin-chains at $Q$ forms a group $\mathtt{G}_{\epsilon}(Q)$. This group is isomorphic to the alternating group $\mathtt{A}_{3}\triangleleft{\mathtt{S}_{3}}$.
\end{corollary}
\begin{proof}
It is enough to notice that for any edge $QR\subset{\mathcal{S}_{P}}$ and for faces $F_{1}$ and $F_{2}$ with $F_{1}\cap{F_{2}}=Q_{1}Q_{2}$ the composition $\sigma^{F_{2}}_{RQ}\circ{\sigma^{F_{1}}_{QR}}$ forms a generator of the alternating group $\mathtt{A}_{3}$ acting on the set $\widehat{Q}$. Since for any chain $\mathsf{W}_{Q}$ the permutation $\sigma^{\mathsf{W}_{Q}}$ must be even the corollary follows.
\end{proof}

\section{Standard Spin Groups $\mathtt{G}(g)$; $g\geq{3}$}
Let $P$ be a standard point of a hyperelliptic Riemann surface $\Sigma$ with genus $g\geq{3}$. Suppose that the vertices of the spin graph $\mathcal{S}^{\epsilon}_{P}$ are  $\{\mathcal{S}_{P}\}=\{P,P_{k},\widetilde{P},\widetilde{P_{k}}; k=1,\ldots,g\}$.  Analogously to a case of genus $3$, we will draw this graph in such a way that te indices of vertices $P_{k}$ (connected with the point $P$) increase from the bottom up (and hence the indices of points $\widetilde{P_{k}}$ connected with the vertex $\widetilde{P}$ decrease from the bottom up).  Of course, our final results do not depend neither on an enumeration of vertices of $\mathcal{S}^{\epsilon}_{P}$  nor on our convention to draw the graph $\mathcal{S}^{\epsilon}_{P}$.  A standard spin graph on a surface of genus $g=4$ is given by Pict3.

\begin{pspicture}(-4.5,-5)(5,5)
\pspolygon[showpoints=true]%
(-2,-3.464)(2,-3.464)(4,0)(2,3.464)(-2,3.464)(-4,0)
\rput(0,-4.6){\rnode{A}{Pict.3}}
\rput(0,-4,2){\rnode{F}{$\mathcal{S}(4)$}}
\psline(-4,0)(-2,1.136)
\psline(-4,0)(-2,-1.164)
\psline(2,1.136)(4,0)
\psline(2,-1.164)(4,0)
\psline[showpoints=true]%
(-2,-1.164)(2,-1.164)
\psline[showpoints=true]%
(-2,1.136)(2,1.136)
\psline[linestyle=dashed]%
(-2,3.464)(2,1.136)
\psline(-2,-3.464)(2,-1.164)
\psline(-2,1.136)(2,3.464)
\psline[linestyle=dashed]%
(-2,3.464)(2,-1.164)
\psline[linestyle=dashed]%
(-2,1.136)(2,-3.464)
\psline[linestyle=dashed]%
(-2,-1.164)(2,3.464)
\psline[linestyle=dashed]%
(-2,-1.164)(2,-3.464)
\psline[linestyle=dashed]%
(-2,-3.464)(2,1.136)

\rput(-4.2,0.1){\rnode{a}{$P$}}
\rput(4.2,0.1){\rnode{b}{$\widetilde{P}$}}
\rput(-2.4,-3.3){\rnode{c}{$P_{1}$}}
\rput(2.4,-3.3){\rnode{d}{$\widetilde{P_{4}}$}}
\rput(-2,1.45){\rnode{e}{$P_{3}$}}
\rput(2,1.45){\rnode{i}{$\widetilde{P_{2}}$}}
\rput(-2,-0.8){\rnode{f}{$P_{2}$}}
\rput(2,-0.8){\rnode{j}{$\widetilde{P_{3}}$}}
\rput(-2.4,3.5){\rnode{g}{$P_{4}$}}
\rput(2.4,3.5){\rnode{h}{$\widetilde{P_{1}}$}}
\end{pspicture}

Similarly to a case of genus $g=3$, for each vertex $Q$ of $\mathcal{S}_{P}$ we introduce the set $\widehat{Q}$ as
\begin{equation*}
\widehat{Q}:=\Phi_{Q}(\Sigma)\cap{\Theta_{\epsilon}}\cong{\{\frac{Q_{1}}{Q},\frac{Q_{2}}{Q},\ldots,\frac{Q_{g}}{Q}\}}\Leftrightarrow{\{1,2,\ldots,g\}_{Q}}
\end{equation*}
where the enumeration of elements $Q_{i}$'s is uniquely determined by the fixed enumeration of points $P_{k}$ of te divisor $\mathcal{A}_{P}$.

Let $F$ be any (quadrangle) face of the graph $\mathcal{S}_{P}$. Suppose that its vertices are $\{F\}=\{P_{n_{1}},P_{n_{2}},\widetilde{P_{n_{3}}},\widetilde{P_{n_{4}}}\}$ with $n_{i}\in{\{1,\ldots,g\}}$ for $i=1,2,3,4$.  We see that any face $F$ determines a unique $3$-cell $\Gamma(F)$ whose vertices are $\{\Gamma(F)\}=\{P_{n_{i}},\widetilde{P_{n_{i}}} ; i=1,\ldots,4\}$. Now, the  unique bijections $\widehat{F}$ between the sets
\begin{equation}
\widehat{F}: \widehat{P_{n_{1}}}\leftrightarrow{\widehat{P_{n_{2}}}}\leftrightarrow{\widehat{\widetilde{P_{n_{3}}}}}\leftrightarrow{\widehat{\widetilde{P_{n_{4}}}}}
\end{equation}
may be defined in the following way: For each $k\notin{\{n_{1},n_{2},n_{3},n_{4}\}}$, $k\in{\{0,1,\ldots,g\}}$, with $P_{0}\equiv{P}$, we will identify the elements
\begin{equation}
\frac{\widetilde{P_{k}}}{P_{n_{1}}}\leftrightarrow \frac{\widetilde{P_{k}}}{P_{n_{2}}}\leftrightarrow \frac{P_{k}}{\widetilde{P_{n_{3}}}}\leftrightarrow \frac{P_{k}}{\widetilde{P_{n_{4}}}}
\end{equation}
 The correspondences between the remaining elements of $\widehat{P_{n_{i}}}$, $i=1,2$, and $\widehat{\widetilde{P_{n_{k}}}}$,  $k=3,4$, are uniquely  determined by the cell $\Gamma(F)\cong{\mathcal{S}(3)}$  and by its face $F$ in the way described in the previous section. (We proceed exactly in the same way  when a face $F$ has vertices $\{F\}=\{P,P_{n_{1}},\widetilde{P_{n_{2}}},P_{n_{3}}\}$.)

Hence, for any oriented edge, say for $P_{n_{i}}{\widetilde{P_{n_{k}}}}$ or for any path from $P_{n_{i}}$ to ${\widetilde{P_{n_{k}}}}$ along the perimeter of $F$ the bijections $\widehat{F}$ given by $(4.1)$ determine an isomorphism ${\widehat{F}}:{\widehat{P_{n_{i}}}}\rightarrow{\widehat{P_{n_{k}}}}$.   Since for any vertex $Q\in{\{\mathcal{S}_{P}\}}$)  we may  identify the set $\widehat{Q}$ with the set $\{1,2,\ldots,g\}_{Q}$, this isomorphism  may be represented by a permutation  $\sigma^{F}_{P_{n_{i}}\widetilde{P_{n_{k}}}}:\{1,2,\ldots,g\}_{P_{n_{i}}}\rightarrow{\{1,2,\ldots,g\}_{\widetilde{P_{n_{k}}}}}$.

   In this way  we obtain that  our face $F\subset{\mathcal{S}_{P}}$ determines  unique permutations
\begin{equation}
\sigma^{F}_{P_{n_{i}}\widetilde{P_{n_{k}}}}, \quad \sigma^{F}_{P_{n_{1}}P_{n_{2}}}, \quad \sigma^{F}_{\widetilde{P_{n_{3}}}\widetilde{P_{n_{4}}}}
\end{equation}
and their inverses that represent appropriate isomorphisms given by $(4.1)$. So, for example
\begin{equation}
\sigma^{F}_{P_{n_{i}}\widetilde{P_{n_{k}}}}=\begin{pmatrix}1&2&{\ldots}&g\\\sigma^{F}(1)&\sigma^{F}(2)&{\ldots}&\sigma^{F}(g)\end{pmatrix}
\end{equation}
where the bottom index $P_{n_{i}}\widetilde{P_{n_{k}}}$ on the right side is omitted.   Analogously for the remaining permutations given by $(4.3)$.

To find these permutations we will introduce some additional objects. Namely, since for any vertex $Q\in{\mathcal{S}_{P}}$ and for any face $F$ with a vertex $Q$ we have a unique $3$-cell $\Gamma={\Gamma(F)}$, we may define   the ordered set $\widehat{Q}^{\Gamma}$ as $\widehat{Q}^{\Gamma}:=\widehat{Q}\cap{\Phi}_{Q}(\Gamma(F))$. This set of three elements  will be identify with the set $\{\overline{1},\overline{2},\overline{3}\}^{\Gamma}_{Q}$. More precisely we have
\begin{equation*}
\{\overline{1},\overline{2},\overline{3}\}^{\Gamma}_{Q}\cong{\widehat{Q}^{\Gamma}}\cong{\{\frac{Q_{m_{1}}}{Q},\frac{Q_{m_{2}}}{Q},\frac{Q_{m_{3}}}{Q}\}}\cong{\{m_{1},m_{2},m_{3}\}}\subset{\{1,2,\ldots,g\}_{Q}}
\end{equation*}
where  $\{Q,Q_{m_{j}},\widetilde{Q},\widetilde{Q_{m_{j}}}; j=1,2,3\}\subset\{\mathcal{S}_{P}\}$ is the set of vertices of the $3$-cell $\Gamma(F)$ and $m_{1}<m_{2}<m_{3}$.

\begin{lemma}
Suppose that $Q_{1}$ and $Q_{2}$ are the vertices of an edge of a standard spin graph $\mathcal{S}_{P}$. For all faces $F\subset{\mathcal{S}_{P}}$ with an edge $Q_{1}Q_{2}$ the permutations  $\sigma^{F}_{Q_{1}Q_{2}}:\{1,2,\ldots,g\}_{Q_{1}}\rightarrow{\{1,2,\ldots,g\}_{Q_{2}}}$ have exactly the same parity.
\end{lemma}

\begin{proof}
Let $\Gamma=\Gamma(F)\cong{\mathcal{S}(3)}$ be the unique $3$-cell determined by a face $F$. It was shown explicitly  that when the genus of a surface is $g=3$ then this lemma is true. For an arbitrary genus $g$ we will show this by considering the following cases:

$\mathbf{CASE}$ $\mathbf{I}$.  Let us consider an edge $PP_{k}$, $k=1,2,\ldots,g$. Now, any face $F\subset{\mathcal{S}_{P}}$ with this edge determines a $3$-cell $\Gamma(F)$ whose set of   vertices is $\{\Gamma(F)\}={\{P,\widetilde{P},P_{n_{j}},\widetilde{P_{n_{j}}}}; j=1,2,3\}$ and $k\in{\{n_{1},n_{2},n_{3}\}}$.  Let us assume that,  according to our convention, we  have $1\leq{n_{1}}<n_{2}<n_{3}\leq{g}$ so that  the graph of $\Gamma(F)$ is given by Pict4.

\begin{pspicture}(-3,-3)(4,3)
\psline[showpoints=true]%
(-1,-1)(1,-1)(2,0)
\psline[showpoints=true]%
(-1,1)(1,1)(2,2)
\psline[showpoints=true, linestyle=dashed]%
(0,2)(0,0)
\psline(-1,-1)(-1,1)
\psline(1,-1)(1,1)
\psline(-1,1)(0,2)
\psline[linestyle=dashed]%
(-1,-1)(0,0)
\psline[linestyle=dashed]%
(0,0)(2,0)
\psline(0,2)(2,2)
\psline(2,0)(2,2)

\rput(0,-2.5){\rnode{A}{Pict.4}}
\rput(-0.5,-1.5){\rnode{B}{$\Gamma(F);$}}

\rput(-1.3,-1.1){\rnode{a}{$P$}}
\rput(1.4,-1.1){\rnode{b}{$P_{n_{1}}$}}
\rput(-0.4,0.1){\rnode{c}{$P_{n_{2}}$}}
\rput(-1.4,1.1){\rnode{d}{$P_{n_{3}}$}}
\rput(-0.4,2.1){\rnode{e}{${\widetilde{P_{n_{1}}}}$}}
\rput(0.9,1.4){\rnode{f}{$\widetilde{P_{n_{2}}}$}}
\rput(2.2,2.1){\rnode{g}{$\widetilde{P}$}}
\rput(2.3,0.2){\rnode{h}{$\widetilde{P_{n_{3}}}$}}

\rput(2.5,-2){\rnode{C}{$1\leq{n_{1}}<{n_{2}}<{n_{3}}\leq{g}$}}

\end{pspicture}

 Suppose that $k=n_{1}$. We have 
\begin{equation*}
\{\overline{1},\overline{2},\overline{3}\}^{\Gamma}_{P}\cong{\{\frac{P_{n_{1}}}{P},\frac{P_{n_{2}}}{P},\frac{P_{n_{3}}}{P}\}}\Leftrightarrow{\{n_{1},n_{2},n_{3}\}}\subset{\{1,2,\ldots,g\}_{P}}
\end{equation*}
and
\begin{equation*}
\{\overline{1},\overline{2},\overline{3}\}^{\Gamma}_{P_{n_{1}}}\cong{\{\frac{P}{P_{n_{1}}},\frac{\widetilde{P_{n_{2}}}}{P_{n_{1}}},\frac{\widetilde{P_{n_{3}}}}{P_{n_{1}}}\}}\Leftrightarrow{\{1,n_{2},n_{3}\}}\subset{\{1,2,\ldots,g\}_{P_{n_{1}}}} 
\end{equation*}
  Let $\mathfrak{F}$ be a face of the cell $\Gamma(F)=\Gamma(\mathfrak{F})$ with an edge $PP_{n_{1}}$. We  know from the previous section that the permutation
\begin{equation*}
\sigma^{\Gamma(\mathfrak{F})}_{PP_{n_{1}}}:\{\overline{1},\overline{2},\overline{3}\}^{\Gamma}_{P}\rightarrow{\{\overline{1},\overline{2},\overline{3}\}^{\Gamma}_{P_{n_{1}}}} 
\end{equation*}
must be even. For example, suppose that  the face $\mathfrak{F}$ has vertices $\{P,P_{n_{1}},\widetilde{P_{n_{3}}},P_{n_{2}}\}$).  In this case we have  
  $\sigma^{\Gamma(\mathfrak{F})}_{PP_{n_{1}}}=\begin{pmatrix}\overline{1}&\overline{2}&\overline{3}\\\overline{3}&\overline{1}&\overline{2}\end{pmatrix}\Leftrightarrow{\begin{pmatrix}n_{1}&n_{2}&n_{3}\\n_{3}&1&n_{2}\end{pmatrix}}$.
 
  Now, the permutation $\sigma^{\mathfrak{F}}_{PP_{n_{1}}}:\{1,2,\ldots,g\}_{P}\rightarrow{\{1,2,\ldots,g\}_{P_{n_{1}}}}$ is uniquely determined by $\sigma^{\Gamma(\mathfrak{F})}_{PP_{n_{1}}}$ together with the mappings
$\frac{P_{i}}{P}\rightarrow\frac{\widetilde{P_{i}}}{P_{n_{1}}}$ for $ i\notin{\{n_{1},n_{2},n_{3}\}}$,  $i\in{\{1,2,\ldots,g\}}$.   
 Hence, the isomorphism  ${\widehat{\mathfrak{F}}}:{\widehat{P}}\rightarrow{\widehat{P_{n_{1}}}}$ is represented by the permutation $\sigma^{\mathfrak{F}}_{PP_{n_{1}}}:{\{1,2,\ldots,g\}_{P}}\rightarrow{\{1,2,\ldots,g\}_{P_{n_{1}}}}$ which has the expilicit form as  $\sigma^{\mathfrak{F}}_{PP_{n_{1}}}=(12{\ldots}{n_{1}n_{3}n_{2}})$. The parity of this cycle is $(-1)^{n_{1}+1}=(-1)^{n_{1}-1}$.
 
 When $k=n_{2}$  then, for any face $\mathfrak{F}_{1}\subset{\Gamma(F)}=\Gamma(\mathfrak{F}_{1})$  with an edge $PP_{n_{2}}$, the permutation $\sigma^{\Gamma(\mathfrak{F}_{1})}_{PP_{n_{2}}}$ must be odd. Suppose that $\mathfrak{F}_{1}=PP_{n_{2}}\widetilde{P_{n_{1}}}P_{n_{3}}$.Since now we have
 \begin{equation*} \{\overline{1},\overline{2},\overline{3}\}^{\Gamma}_{P_{n_{2}}}\cong{\{\frac{P}{P_{n_{2}}},\frac{\widetilde{P_{n_{1}}}}{P_{n_{2}}},\frac{\widetilde{P_{n_{3}}}}{P_{n_{2}}}\}}\Leftrightarrow{\{1,n_{1}+1,n_{3}\}}\subset{\{1,2,\ldots,g\}_{P_{n_{2}}}}
 \end{equation*}
 we obtain that the appropriate permutation is
 \begin{equation*} \sigma^{\Gamma(\mathfrak{F}_{1})}_{PP_{n_{2}}}=\begin{pmatrix}\overline{1}&\overline{2}&\overline{3}\\\overline{3}&\overline{2}&\overline{1}\end{pmatrix}\Leftrightarrow{\begin{pmatrix}n_{1}&n_{2}&n_{3}\\n_{3}&n_{1}+1&1\end{pmatrix}}
 \end{equation*}
 . The corresponding isomorphism ${\widehat{\mathfrak{F}}_{1}}$ is now represented by a permutation $\sigma^{\mathfrak{F}_{1}}_{PP_{n_{2}}}:\{1,2,\ldots,g\}_{P}\rightarrow{\{1,2,\ldots,g\}_{P_{2}}}$  which can be written as the product of two cycles as follows: $\sigma^{\mathfrak{F}_{1}}_{PP_{n_{2}}}=(1{\ldots}(n_{1}-1)n_{1}n_{3})((n_{1}+1){\ldots}n_{2})$. We see that 
 \begin{equation}
 sgn(\sigma^{\mathfrak{F}_{1}}_{PP_{n_{2}}})=(-1)^{n_{1}+(n_{2}-n_{1}-1)}=(-1)^{n_{2}-1}
 \end{equation}
We may see that the other face of $\Gamma(F)$ with the edge $PP_{n_{2}}$ produces a permutation whose parity is also given by $(-1)^{n_{2}-1}$.
 
 Analogously, when $k=n_{3}$ any face $\mathfrak{F}_{2}\subset{\Gamma(F)}=\Gamma(\mathfrak{F}_{2})$ with an edge $PP_{n_{3}}$ determines a permutation $\sigma^{\Gamma(\mathfrak{F}_{2})}_{PP_{n_{3}}}:\{\overline{1},\overline{2},\overline{3}\}_{P}\rightarrow{\{\overline{1},\overline{2},\overline{3}\}_{P_{n_{3}}}}$ which must be even.  It is easy to see that  the corresponding isomorphism $\sigma^{\mathfrak{F}_{2}}_{PP_{n_{3}}}:\{1,2,\ldots,g\}_{P}\rightarrow{\{1,2,\ldots,g\}_{P_{n_{3}}}}$ is represented by a permutation which has the parity equal to $(-1)^{n_{3}-1}$.
 
 Summarizing, any face $F\subset{\mathcal{S}^{\epsilon}_{P}}$ with an edge $PP_{k}$ produces a unique isomorphism $\sigma^{F}_{PP_{k}}:\{1,2,\ldots,g\}_{P}\rightarrow{\{1,2,\ldots,g\}_{P_{k}}}$ which is represented by a permutation whose parity depends only on the index $k$, namely:  $sgn(\sigma^{F}_{PP_{k}})=(-1)^{k-1}$.

 $\mathbf{CASE}$ $II$. Let us consider an edge $P_{k}\widetilde{P_{l}}$, $k,l\in{\{1,2,\ldots,g\}}$, $k\neq{l}$.Any face with this edge determines the unique  $3$-cell $\Gamma\subset{\mathcal{S}_{P}}$ whose set of vertices   is equal to    $\{\Gamma\}=\{P_{n_{k}},\widetilde{P_{n_{k}}}; k=1,..,4\}$ and $k,l\in{\{n_{1},n_{2},n_{3},n_{4}\}}$.  Suppose that we have  $1\leq{n_{1}}<n_{2}<n_{3}<n_{4}\leq{g}$.  Now, according to our convention we will draw the graph of $\Gamma\cong{\mathcal{S}(3)}$ as on Pict5.

We see that for the vertex $P_{n_{1}}$  we have
\begin{equation*}
{\{\overline{1},\overline{2},\overline{3}\}^{\Gamma}_{P_{n_{1}}}}\cong{\{\frac{\widetilde{P_{n_{2}}}}{P_{n_{1}}},\frac{\widetilde{P_{n_{3}}}}{P_{n_{1}}},\frac{\widetilde{P_{n_{4}}}}{P_{n_{1}}}\}}\Leftrightarrow{\{n_{2},n_{3},n_{4}\}}\subset{\{1,2,\ldots,g\}_{P_{n_{1}}}}
\end{equation*}
and (for example) for the vertex $\widetilde{P_{n_{4}}}$ we have
\begin{equation*}
{\{\overline{1},\overline{2},\overline{3}\}^{\Gamma}_{\widetilde{P_{n_{4}}}}}\cong{\{\frac{P_{n_{1}}}{\widetilde{P_{n_{4}}}},\frac{P_{n_{2}}}{\widetilde{P_{n_{4}}}},\frac{P_{n_{3}}}{\widetilde{P_{n_{4}}}}\}}\Leftrightarrow{\{n_{1}+1,n_{2}+1,n_{3}+1\}}\subset{\{1,2,\ldots,g\}_{\widetilde{P_{n_{4}}}}}
\end{equation*}
From Pict5 and lemma3 we have that any face $\mathfrak{F}\subset{\Gamma}$ produces a permutation  $\sigma^{\Gamma(\mathfrak{F})}_{P_{n_{1}}{\widetilde{P_{n_{4}}}}}$ which must be even. For example, let $\mathfrak{F}=\{P_{n_{1}}\widetilde{P_{n_{4}}}P_{n_{2}}\widetilde{P_{n_{3}}}\}$. The isomorphism ${\widehat{\mathfrak{F}}}$ from ${\widehat{P_{n_{1}}}}^{\Gamma}$ to the set ${\widehat{\widetilde{P_{n_{4}}}}}^{\Gamma}$ corresponds to the permutation $\sigma^{\Gamma(\mathfrak{F})}_{P_{n_{1}}{\widetilde{P_{n_{4}}}}}=\begin{pmatrix}\overline{1},&\overline{2}&\overline{3}\\\overline{3}&\overline{1}&\overline{2}\end{pmatrix}\Leftrightarrow{\begin{pmatrix}n_{2}&n_{3}&n_{4}\\n_{3}+1&n_{1}+1&n_{2}+1\end{pmatrix}}$. This permutation together with the mappings
\begin{equation*}
\frac{\widetilde{P_{n_{i}}}}{P_{n_{1}}}\rightarrow{\frac{P_{i}}{\widetilde{P_{n_{4}}}}} \quad \text{for} \quad i\in{\{0,1,\ldots,g\}}, \quad  i\notin{\{n_{1},n_{2},n_{3},n_{4}\}},\quad P_{0}=P
\end{equation*} 
determines a unique isomorphism
\begin{equation*}
 \sigma^{\mathfrak{F}}_{P_{n_{1}}\widetilde{P_{n_{4}}}}:\{1,2,\ldots,g\}_{P_{n_{1}}}\rightarrow{\{1,2,\ldots,g\}_{\widetilde{P_{n_{4}}}}}
 \end{equation*}
 which correspons to the permutation
 \begin{equation*}
 \sigma^{\mathfrak{F}}_{P_{n_{1}}\widetilde{P_{n_{4}}}}=(n_{1}+1,n_{1}+2,\ldots,n_{2},n_{3}+1, n_{3}+2,\ldots,n_{4},n_{2}+1,n_{2}+2,\ldots,n_{3})
 \end{equation*}
 The signature of this permutation is 
 \begin{equation}
 sgn(\sigma^{\mathfrak{F}}_{P_{n_{1}}\widetilde{P_{n_{4}}}})=(-1)^{n_{4}-n_{1}-1}
 \end{equation}

 \begin{pspicture}(-3,-3)(4,3)
 \psline[showpoints=true]%
 (-1,-1)(1,-1)(2,0)
 \psline[showpoints=true]%
 (-1,1)(1,1)(2,2)
 \psline(-1,1)(0,2)
 \psline(-1,1)(1,1)
 \psline(-1,-1)(-1,1)
 \psline(0,2)(2,2)
 \psline(2,0)(2,2)
 \psline(1,-1)(1,1)
 \psline[showpoints=true, linestyle=dashed]%
 (0,2)(0,0)
 \psline[linestyle=dashed]%
 (-1,-1)(0,0)
 \psline[linestyle=dashed]%
 (0,0)(2,0)
\rput(-1.4,-1.1){\rnode{a}{$P_{n_{1}}$}} 
\rput(1.4,-1.1){\rnode{b}{$\widetilde{P_{n_{4}}}$}}
\rput(2.4,0.2){\rnode{c}{$P_{n_{2}}$}}
\rput(2.4,2.2){\rnode{d}{$\widetilde{P_{n_{1}}}$}}
\rput(-0.4,2.2){\rnode{e}{$P_{n_{4}}$}}
\rput(-1.4,1.1){\rnode{f}{$\widetilde{P_{n_{2}}}$}}
\rput(0.8,1.3){\rnode{g}{$P_{n_{3}}$}}
\rput(-0.4,0.2){\rnode{h}{$\widetilde{P_{n_{3}}}$}}

 \rput(0,-2){\rnode{A}{Pict.5}}
 \end{pspicture}

 The same face $\mathfrak{F}$ determines the isomorphism $\widehat{P_{n_{1}}}^{\Gamma}\rightarrow{\widehat{\widetilde{P_{n_{3}}}}^{\Gamma}}$ where 
 \begin{equation*} \widehat{\widetilde{P_{n_{3}}}}^{\Gamma}\cong{\{\overline{1},\overline{2},\overline{3}\}^{\Gamma}_{\widetilde{P_{n_{3}}}}}\cong{\{\frac{P_{n_{1}}}{\widetilde{P_{n_{3}}}},\frac{P_{n_{2}}}{\widetilde{P_{n_{3}}}},\frac{P_{n_{4}}}{\widetilde{P_{n_{3}}}}\}}\Leftrightarrow{\{n_{1}+1,n_{2}+1,n_{4}\}}\subset{\{1,2,\ldots,g\}_{\widetilde{P_{n_{3}}}}}
 \end{equation*}
 This mapping corresponds to an odd permutation 
 \begin{equation*} \sigma^{\Gamma(\mathfrak{F})}_{P_{n_{1}}\widetilde{P_{n_{3}}}}=\begin{pmatrix}\overline{1}&\overline{2}&\overline{3}\\\overline{3}&\overline{2}&\overline{1}\end{pmatrix}\Leftrightarrow{\begin{pmatrix}n_{2}&n_{3}&n_{4}\\n_{4}&n_{2}+1&n_{1}+1\end{pmatrix}}
 \end{equation*} 
  In exactly the same way as   before we obtain a unique    permutation
  \begin{equation*} \sigma^{\mathfrak{F}}_{P_{n_{1}}\widetilde{P_{n_{3}}}}:\{1,2,\ldots,g\}_{P_{n_{1}}}\rightarrow{\{1,2,\ldots,g\}_{\widetilde{P_{n_{3}}}}}
  \end{equation*}
   which, this time,  is equal to the product of two cycles
  \begin{equation*}
   {\sigma^{\mathfrak{F}}_{P_{n_{1}}\widetilde{P_{n_{3}}}}}=(n_{1}+1,n_{1}+2,..,n_{2},n_{4})(n_{2}+1,n_{2}+2,..,n_{3})
   \end{equation*}
 The signature of this permutation  is 
  \begin{equation}
  sgn(\sigma^{\mathfrak{F}}_{P_{n_{1}}\widetilde{P_{n_{3}}}})=(-1)^{n_{3}-n_{1}-1}
 \end{equation}
  The permutation produced by the other face  $\mathfrak{F_{1}}\subset{\Gamma}$ with the edge $P_{n_{1}}\widetilde{P_{n_{3}}}$ has exactly the same parity (given by $(4.7)$). 
  
  Similarly, we may check that the isomorphism $\widehat{P_{n_{1}}}^{\Gamma}\rightarrow{\widehat{\widetilde{P_{n_{2}}}}}^{\Gamma}$ produced by any  face of $\Gamma$ with an edge $P_{n_{1}}\widetilde{P_{n_{2}}}$ corresponds to an even permutation.  It induces  an isomorphism
  \begin{equation*}
   \{1,2,\ldots,g\}_{P_{n_{1}}}\rightarrow{\{1,2,\ldots,g\}_{\widetilde{P_{n_{2}}}}}
   \end{equation*}
    which is represented by a permutation whose parity is given by $(-1)^{n_{2}-n_{1}-1}$.
  
 We obtain analogous results for any edge $P_{l}\widetilde{P_{k}}\subset{\Gamma}\subset{\mathcal{S}_{P}}$ when the points $P$ and $\widetilde{P}$ are vertices of $\Gamma$ (i.e. when some $n_{j}={0}$ and $P_{0}=P$).
 
 Hence,  using the fact that the parities of permutation and its inverse coincide, we see that the signature of  the  permutation which represents an isomorphism $\sigma^{F}_{Q_{1}Q_{2}}:\{1,2,\ldots,g\}_{Q_{1}}\rightarrow{\{1,2,\ldots,g\}_{Q_{2}}}$    depends only on the vertices $Q_{1}$ and $Q_{2}$ and not on the choice of a face $F\subset{\mathcal{S}_{P}}$ along the edge $Q_{1}Q_{2}$.
\end{proof}

Let $\mathsf{W}_{Q}$ be any spin-chain at a vertex Q which was defined earlier by the definition 4.
\begin{lemma}
 Any spin-chain $\mathsf{W}_{P}=(\mathsf{L}_{P}=PQ_{1}{\ldots}Q_{N}P;F_{1},F_{2},\ldots,F_{N+1})$ determines a unique permutation $\sigma^{\mathsf{W}_{P}}$ of the set $\widehat{P}\cong{\mathcal{A}_{P}}$ which is an even permutation.
\end{lemma}

\begin{proof}
 It is enough to consider only simple loops $\mathsf{L}_{P}$.

$\mathbf{CASE}$ $\mathbf{I}$: Assume that the conjugate $\widetilde{P}$ is not a vertex of $\mathsf{L}_{P}$.  In this case we must have 
\begin{equation*}
\mathsf{L}_{P}=PP_{n_{1}}\widetilde{P_{m_{1}}}{\ldots}P_{n_{k}}\widetilde{P_{m_{k}}}P_{n_{k+1}}P
\end{equation*}
with ${\widetilde{P_{m_{p}}}}\neq{\widetilde{P_{m_{q}}}}$ and with $P_{n_{p}}\neq{P_{n_{q}}}\neq{P_{n_{k+1}}}$ for $p,q\in{\{1,2,,,,k\}}$,     $p\neq{q}$. Hence, ${N+1}=2k+2$ is even. Using considerations similar to those in lemma $5$  it is easy to see that $sgn(\sigma^{\mathsf{W}_{P}})=(-1)^{2n_{k+1}}=+1$ so that $\sigma^{\mathsf{W}_{P}}$ is an even permutation.

$\mathbf{CASE}$  $\mathbf{II}$. Suppose that $\widetilde{P}$ is a vertex of a loop $\mathsf{L}_{P}$. This means that 
\begin{equation*}
\mathsf{L}_{P}=PP_{n_{1}}\widetilde{P_{m_{1}}}{\ldots}P_{n_{l}}\widetilde{P_{m_{l}}}\widetilde{P}\widetilde{P_{m_{l+1}}}P_{n_{l+1}}{\ldots}\widetilde{P_{m_{k}}}P_{n_{k}}P
\end{equation*}
and our chain is given by $\mathsf{W}_{P}=(\mathsf{L_{P}}; F_{1},\ldots,F_{N+1})$ with ${N+1}=2k+2$. Since the sum: 
$n_{1}+(m_{1}-n_{1})+(n_{2}-m_{1})\ldots+(n_{l}-m_{l-1})+(m_{l}-n_{l})+m_{l}+m_{l+1}+(n_{l+1}-m_{l+1})+(m_{l+2}-n_{l+1})+{\ldots}+(m_{k}-n_{k-1})+(n_{k}-m_{k})+n_{k}$ is equal to $2m_{l}+2n_{k}$ the parity of $\sigma^{\mathsf{W}_{P}}$ ( which is given  by  $sgn(\sigma^{\mathsf{W}_{P}}=(-1)^{2m_{l}+2n_{k}})$) is even. 
\end{proof}

\begin{lemma}
For any vertex $Q$ of a standard spin graph $\mathcal{S}_{P}$ on a surface $\Sigma$  of genus $g\geq{3}$, the set of all spin-chains at $Q$ produces a group $\mathtt{G}_{Q}$ of permutations of the set $\widehat{Q}\cong{\{\mathcal{A}_{Q}\}}$ which is isomorphic to the alternating group $\mathtt{A}_{g}\triangleleft{\mathtt{S}_{g}}$.
\end{lemma}
\begin{proof}
It is obvious that the set $\{\sigma^{\mathsf{W}_{Q}}\}$  (where $ \mathsf{W}_{Q}$ is any chain at $Q$)  of permutations of the set $\widehat{Q}$ forms a group. Since each of these permutations must be even and since the chains of the form $\mathsf{W}_{Q}=\{\mathsf{L}_{Q}=QSQ; F_{1},F_{2}\}$ for $S\in\{\mathcal{A}_{Q}\}$ and with  $F_{1}\neq{F_{2}}$ produce the generators of the alternating group $\mathtt{A}_{g}$ acting on the set $\{\mathcal{A}_{Q}\}\cong{\widehat{Q}}\cong{\{1,2,\ldots,g\}_{Q}}$,  we must have $\mathtt{G}_{Q}\cong{\mathtt{A}_{g}}$.
\end{proof}

\section{Conjugate Spin Groups}
Let $Q$ be a vertex of some standard spin graph $\mathcal{S}^{\epsilon}_{P}$ on a surface $\Sigma$ of genus  $g\geq{3}$. By $\mathfrak{W}_{Q}$ we will denote the set of all spin chains at $Q$.  Equivalently, $\mathfrak{W}_{Q}$ is the set of all sequences of Jacobi rotations $\Phi_{RS}$ ( here $RS$ is en edge of $\mathcal{S}^{\epsilon}_{P}$) from $\Phi_{Q}$ to $\Phi_{Q}$,  together with all possible choices of spin permutations for each Jacobi rotation (i.e. together with all possible choices of faces along the edges of a loop).

For any two vertices $Q$ and $R$ of the spin graph $\mathcal{S}^{\epsilon}_{P}$ to construct an isomorphism $f:\mathfrak{W}_{Q}\rightarrow{\mathfrak{W}_{R}}$  and then an isomorphism    $\varphi:\mathtt{G}_{Q}\rightarrow{\mathtt{G}_{R}}$   we may take any (oriented) path in $\mathcal{S}_{P}$ from $Q$ to $R$ together with a choice of faces along the edges of this path.  
 
It appears that for any $Q,R\in{\{\mathcal{S}_{P}\}}$ we may uniquely define a shortest path $C_{QR}$  between these vertices as well as there is a uniquely defined choice for faces along the edges of this path.  

First we notice, that for any two vertices $Q$ and $R$ we have  the following three possibilities:
\begin{itemize}
\item $Q$ and $R$ are connected by an edge $QR$. In this case we obviously  have: $C_{QR}=QR$. 
\item $Q$ and $R$ are not vertices of the same edge but $R\neq{\widetilde{Q}}$. Now $C_{QR}=QSR$ for some vertex $S$ which is connected  to both, to  $Q$ and to $R$.
\item the vertex $R={\widetilde{Q}}$. In this case we have $C_{QR}=QS_{1}S_{2}R$ for some vertex $S_{1}$ connected as well to $Q$ as to $S_{2}$, and with $S_{2}$  connected also to $R$. 
\end{itemize}
Since the standard graph is totally symmetric with respect to all of its vertices and since our fixed enumerations of the points of the integral divisor $\mathcal{A}^{\epsilon}_{P}$ induces (for each vertex $Q$) the unique enumeration  of the elements of $\{\mathcal{A}^{\epsilon}_{Q}\}$,   we obtain immediately that 
\begin{itemize}
\item When $QR$ is an edge of the graph $\mathcal{S}^{\epsilon}_{P}$ then we have $R=Q_{i}$ for some $i=1,2,..,g$ and $C_{QR}=QQ_{i}$.
\item When $Q$ and $R$ are not connected but $R\neq{\widetilde{Q}}$ then we must have $R={\widetilde{Q}_{j}}$ for some $j=1,2,..,g$. Now we may take $C_{QR}=QQ_{i}\widetilde{Q}_{j}$ with  any $i\neq{j}$. By the requirement that $i=\min\{1,2,..,\widehat{j},..g\}$ we obtain  the unique  shortest path from $Q$ to $R$ (here the hat over $j$ means that $j$ is omitted).
\item When $R={\widetilde{Q}}$ then any path $C_{QR}=QQ_{i}\widetilde{Q}_{j}\widetilde{Q}$, $i\neq{j}$,  in $\mathcal{S}^{\epsilon}_{P}$ is the shortest path that connects $Q$ and $R$.  We always can  require that $i=1$ and $j=2$ to obtain a uniduely defined shortest path.
\end{itemize}
Since the order of writing the vertices of $C_{QR}$ indicates the direction of traveling along this path, we will consider $C_{RQ}$ as the same path as $C_{QR}$ but travelled in the opposite direction, (i.e. from $R$ to $Q$).  Now, the one-to-one correspondence from the set of loops at $Q$ to the set of all loops at $R$ will be denoted by $\mathcal{A}{d}(C_{QR})$. It is given as
\begin{equation*}
\mathcal{A}{d}(C_{QR}): \mathsf{L}_{Q}\rightarrow{\mathsf{L}_{R}}={C_{QR}}{\mathsf{L}_{Q}}C_{RQ}
\end{equation*}
It only remains to fix faces along the edges of the path $C_{QR}$. Of course, any choice of faces will work.  However, our unique form of the path $C_{QR}$ for each possible  case of $R$, that is for $R=Q_{i}$ or for  $R=\widetilde{Q}_{j}$  or for $R=\widetilde{Q}$, allows us to consider a unique choice for faces along $C_{QR}$ in the following way:
\begin{enumerate}
\item When $R=Q_{i}$ then $C_{QR}=QQ_{i}$ and we will consider the face $\overline{F}=QQ_{i}\widetilde{Q}_{k}Q_{l}$ with $l=\min\{1,2,..,\widehat{i},..g\}$  and with $k=\min\{1,2,..\widehat{i},..,\widehat{l},..g\}$ when $i<l$ and similarly when $i>l$.
\item When $R=\widetilde{Q}_{j}$ then $C_{QR}=QQ_{i}\widetilde{Q}_{j}$ with $i$ determined  in the definition of the path $C_{QR}$.  Now, we take the face $\overline{F}_{1}$ along $QQ_{i}$ as $\overline{F}_{1}=\overline{F}$ given by $(1)$ above and a face $\overline{F}_{2}$ along $Q_{i}\widetilde{Q}_{j}$ given by vertices $\overline{F}_{2}=Q_{i}\widetilde{Q}_{j}\widetilde{Q}\widetilde{Q}_{k}$ with $k=\min\{1,2,..,\widehat{i},..,\widehat{j},..g\}$  and similarly when $j<i$.
\item When $R=\widetilde{Q}$, i.e. when $C_{QR}=QQ_{1}\widetilde{Q}_{2}\widetilde{Q}$, then along the edge $QQ_{1}$ we consider the face $\overline{F}_{1}=QQ_{1}\widetilde{Q}_{2}Q_{3}$, along the edge $Q_{1}\widetilde{Q}_{2}$ we will take the face $\overline{F}_{2}=Q_{1}\widetilde{Q}_{2}\widetilde{Q}\widetilde{Q}_{3}$ and along the edge $\widetilde{Q}_{2}\widetilde{Q}$ we consider the face $\overline{F}_{3}=\widetilde{Q}_{2}\widetilde{Q}\widetilde{Q}_{1}Q_{3}$
\end{enumerate} 
Now, when for example $R={\widetilde{Q}_{j}}$, the bijection $\widehat{C}_{QR}$ from the set $\mathfrak{W}_{Q}$ of all spin chains at $Q$ to the set $\mathfrak{W}_{R}$ of all spin chains at $R$  is  given by
\begin{equation*}
\widehat{C}_{QR}:\mathsf{W}_{Q}=\{\mathsf{L}_{Q};F_{1},\ldots,F_{N+1}\}\rightarrow{\{\mathcal{A}d(C_{QR})\mathsf{L}_{Q}; \overline{F}_{1},\overline{F}_{2},F_{1},\ldots,F_{N+1},\overline{F}_{2},\overline{F}_{1}\}}
\end{equation*}
  For the remaining cases of $C_{QR}$ i.e. when $R=Q_{i}$ or when $R=\widetilde{Q}$ we proceed analogously.

Let $\sigma^{C}_{QR}$ denote the permutation corresponding to the unique isomorphism $\widehat{Q}\rightarrow{\widehat{R}}$  determined by the edges of the path $C_{QR}$ and by our choices of faces along these edges.  More precisely we will consider:
\begin{equation*}
\sigma^{C}_{QR}=\sigma^{\overline{F}}_{QQ_{i}}\quad \text{when}\quad R=Q_{i}; \quad \sigma^{C}_{QR}=\sigma^{\overline{F}_{2}}_{Q_{i}\widetilde{Q}_{j}}\sigma^{\overline{F}_{1}}_{Q{Q}_{i}} \quad \text{when}\quad R=\widetilde{Q}_{j}
\end{equation*}  
and $\sigma^{C}_{QR}=\sigma^{\overline{F}_{3}}_{\widetilde{Q}_{2}\widetilde{Q}}\sigma^{\overline{F}_{2}}_{Q_{1}\widetilde{Q}_{2}}\sigma^{\overline{F}_{1}}_{QQ_{1}}$ when  $R=\widetilde{Q}$.

So, for each permutation $\sigma^{\mathsf{W}_{Q}}$ of the set $\widehat{Q}\cong{\{\mathcal{A}_{Q}\}}$ we obtain a well defined permutation of  the set $\widehat{R}\cong{\{\mathcal{A}_{R}\}}$ as :
\begin{equation}
\sigma^{\mathsf{W}_{Q}}\rightarrow{\sigma^{\mathsf{W}_{R}}}=\sigma^{C}_{QR}\sigma^{\mathsf{W}_{Q}}\sigma^{C}_{RQ}
\end{equation}
This means that for any two vertices $Q$ and $R$ of a standard spin graph $\mathcal{S}_{P}$ on $\Sigma$ the spin groups $\mathtt{G}_{Q}$ and $\mathtt{G}_{R}$ are conjugate to each other; $\mathtt{G}_{R}={\mathcal{A}d(\sigma^{C}_{QR})}{\mathtt{G}_{Q}}$. Of course, with any other choice of a path $C$ from $Q$ to $R$ as well as for any different choices of faces along the edges of the path $C$ we will also obtain a bijection between the sets $\mathfrak{W}_{Q}$ and $\mathfrak{W}_{R}$  together with a concrete conjugation of the appropriate groups. However these mappings will have different forms. 

We notice that, although all spin groups, at any standard point of $\Sigma$, are  isomorphic to each other and to the alternating group ${\mathtt{A}_{g}}\triangleleft{\mathtt{S}_{g}}$, we are able to find a concrete isomorphism $\mathcal{A}d(\sigma^{C}_{PQ}):{\mathtt{G}_{P}}\rightarrow{\mathtt{G}_{Q}}$ only when both $P$ and $Q$ are vertices of exactly the same spin graph on $\Sigma$.

When $P$ and $Q$ are vertices of distinct spin graphs then an enumeration of the points of $\{\mathcal{A}_{P}\}$ does not imply any enumeration of the set $\{\mathcal{A}_{Q}\}$.  Hence, to construct any isomorphism between the (naturally attached) groups $\mathtt{G}_{P}$ and $\mathtt{G}_{Q}$ we must first determine a one-one correspondence between the set $\{\mathcal{A}_{P}\}$ and the set $\{\mathcal{A}_{Q}\}$.

\section{Spin Groups at Weierstrass points}
On any hyperelliptic Riemann surface $\Sigma$, any nonsingular even spin bundl $\xi_{\epsilon}$  divides the set $\mathcal{W}$ of all Weierstrass points (of $\Sigma$) into two disjoint subsets. Each such subset has $g+1$ points that form the set of all vertices of an appropriate Weierstrass spin graph. More precisely, $\mathcal{W}={\{\mathcal{S}_{P}\}}\cup{\{\mathcal{S}_{R}\}}$ for some Weierstrass points $P$, $R$ such that $R\notin{\{\mathcal{A}_{P}\}}$. On the contrary to a standard spin graph $\mathcal{S}(g)$  (with $g>2$), where all faces are quadrangles, any Weierstrass spin graph $\mathcal{S}_{P}$ has all faces triangular (for any genus $g\geq{2}$). The examples of such graphs for $g=2$ and $g=3$ are given by the Pict6a and Pict6b respectively.

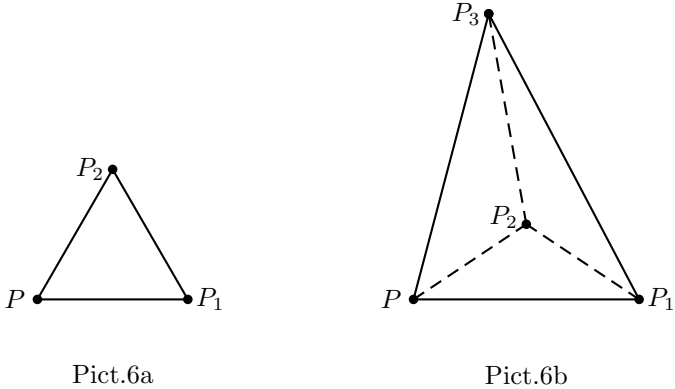
\begin{pspicture}(-4.5,-1.7)(4.5,4.5)
\psline[showpoints=true]%
(-4,0)(-2,0)(-3,1.73)
\psline(-4,0)(-3,1.73)
\rput(-3,-1){\rnode{A}{Pict.6a}}
\rput(-4.3,0){\rnode{a}{$P$}}
\rput(-1.7,0){\rnode{b}{$P_{1}$}}
\rput(-3.3,1.73){\rnode{c}{$P_{2}$}}

\psline[showpoints=true]%
(1,0)(4,0)(2,3.8)
\psline(1,0)(2,3.8)
\psline[showpoints=true, linestyle=dashed]%
(2,3.8)(2.5,1)
\psline[linestyle=dashed]%
(1,0)(2.5,1)
\psline[linestyle=dashed]%
(4,0)(2.5,1)

\rput(2.5,-1){\rnode{B}{Pict.6b}}
\rput(0.7,0){\rnode{d}{$P$}}
\rput(4.3,0){\rnode{e}{$P_{1}$}}
\rput(2.2,1.1){\rnode{f}{$P_{2}$}}
\rput(1.7,3.8){\rnode{g}{$P_{3}$}}
\end{pspicture}

When genus $g={3}$ then, similarly as in a standard case,  we may identify the set $\widehat{Q}\cong{\{\mathcal{A}_{Q}\}}$ with the set $\{1,2,3\}_{Q}$    for any vertex $Q\in{\{\mathcal{S}_{P}\}}\subset{\mathcal{W}}$.

     It is easy to see that  any face $F\subset{\mathcal{S}_{P}}$, $P\in{\mathcal{W}}$ with an edge $QR$ determines  a unique isomorphism $\sigma^{F}_{QR}:{\{1,2,3\}_{Q}}\rightarrow{\{1,2,3\}_{R}}$ in the way analogous to a standard case. Also, when $F_{1}$ is any other face of $\mathcal{S}_{P}$ with an edge $QR$ then the compositions $\sigma^{F}_{QR}\circ{\sigma^{F_{1}}_{RQ}}$ is an even permutation of the set $\widehat{R}\cong{\{\mathcal{A}_{R}\}}$ as well as the composition ${\sigma^{F_{1}}_{RQ}}\circ{\sigma^{F}_{QR}}$ is an even permutation of the set ${\widehat{Q}}$. 

For any Weierstrass point $Q\in{\mathcal{W}}$ we will define the spin group $\mathtt{G}_{Q}$ in exatly the same way as in the case of a standard point of $\Sigma$.  It is easy to see that  when the genus is $g=3$ and $Q\in{\mathcal{W}}$ then  any   spin group $\mathtt{G}_{Q}$ is isomorphic to the alternating group $\mathtt{A}_{3}\triangleleft{\mathtt{S}_{3}}$.

\begin{lemma}
Suppose that $\Sigma$ is a hyperelliptic Riemann surface with the genus $g\geq{3}$. For any Weierstrass point $P\in{\mathcal{W}}\subset{\Sigma}$   the corresponding spin group $\mathtt{G}_{P}$ is isomorphic to the alternating group $\mathtt{A}_{g}\triangleleft{\mathtt{S}_{g}}$.
\end{lemma}
\begin{proof}
Suppose that the set of all vertices of the graph $\mathcal{S}_{P}$ is ${\{\mathcal{S}_{P}\}}={\{P,P_{1},\ldots,P_{g}\}}$. Let us, for example, consider  the edge $PP_{g}\subset{\mathcal{S}_{P}}$. This edge is common for $g-1$ different triangular faces $F_{i}$ with vertices ${\{F_{i}\}}={\{P,P_{i},P_{g}\}}$ respectively; $i=1,2,\ldots,g-1$. We see that any two of these faces, say $F_{i}$ and $F_{j}$, $i\neq{j}$, determine a unique Weierstrass $3$-cell $\Gamma$ that is isomorphic to a Weierstrass spin-graph on a surface of genus $g=3$, (i.e.  to the spin graph given by Pict6b).  The set of vertices of this cell is ${\{\Gamma\}}={\{F_{i}\}\cap{\{F_{j}\}}}={\{P,P_{i},P_{j},P_{g}\}}$, $i,j\in{\{1,\ldots,g-1\}}$; $i\neq{j}$.   The subgroup of permutations of the set $\widehat{P}\cong{\{\mathcal{A}_{P}\}}\Leftrightarrow{\{1,2,\ldots,g\}_{P}}$  produced by this cell is generated by the permutation
\begin{equation}
\begin{pmatrix}1&2&{\ldots}&i&{\ldots}&j&{\ldots}&g\\1&2&{\ldots}&j&{\ldots}&g&{\ldots}&i\end{pmatrix}=(ijg)\in{\mathtt{S}_{g}(\widehat{P})}
\end{equation}
Hence, the set of all spin chains $\mathsf{W}_{P}$ at $P\in{\mathcal{W}}$ (which are defined in exactly the same way as for  standard spin graphs) produces even permutations of the set $\widehat{P}$ that form the spin group $\mathtt{G}_{P}$. From $(6.1)$ it is obvious that this group has to be   isomorphic to the alternating group, i.e. $\mathtt{G}_{P}\cong{\mathtt{A}_{g}}\triangleleft{\mathtt{S}_{g}}$.

Since the Weierstrass graph $\mathcal{S}_{P}$  is totally symetric with respect to all of its vertices, i.e. since the situation of any other vertex $Q\in{\{\mathcal{S}_{P}\}}\subset{\mathcal{W}}$ is exactly the same as one of the point $P$, we obtain immediately that for any $k=1,2,\ldots,g$ we have $\mathtt{G}_{P}\cong{\mathtt{G}_{P_{k}}}\cong{\mathtt{A}_{g}}\triangleleft{\mathtt{S}_{g}}$.
\end{proof}

\section{SUMMARY}

We have shown that a spin-graph  structure of  a standard or Weierstrass leaf of the $\epsilon$-foliation   yields to attaching the spin group $\mathtt{G}_{P}$ at each (standard or Weierstrass) point $P$ of a surface $\Sigma$.  More precisely, if $P$ and $Q$ belong to the same (standard or Weierstrass) leaf of the foliation then the Jacobi rotation $\Phi_{QP}$ from $\Phi_{P}:{\Sigma}\rightarrow{Jac{\Sigma}}$ to $\Phi_{Q}:{\Sigma}\rightarrow{Jac{\Sigma}}$ corresponds to passing from the set $\widehat{P}={{\Phi}_{P}(\Sigma)\cap{\Theta_{\epsilon}}}$ to the set ${\widehat{Q}}={{\Phi}_{Q}(\Sigma)\cap{\Theta_{\epsilon}}}$. Obviously, after a few Jacobi rotations, for example after a sequence ${\Phi}_{PR}{\Phi}_{RQ}{\Phi}_{QP}$, we will return to our original set $\widehat{P}$. Now, a spin-graph structure of a leaf found in $[2]$ allows us to extract  a group action of sequences of Jacobi rotations which start and end with the mapping $\Phi_{P}$. This group, denoted by $\mathtt{G}_{P}$, acts on the set  $\widehat{P}={\Phi_{P}(\{\mathcal{A}_{P}\})}\cong{\{\mathcal{A}_{P}\}}$ and it is isomorphic to the alternarting group ${\mathtt{A}_{g}}\triangleleft{\mathtt{S}_{g}}$. Moreover, for any two points $P$ and $Q$  of the same leaf of the $\epsilon$-foliation there is a natural (adjoint) isomorphism between the groups $\mathtt{G}_{P}$ and $\mathtt{G}_{Q}$. To obtain any isomorphism between the spin group that are attached at points $P$ and $Q$ belonging to  different standard (or Weierstrass) leaves we must first determine a one one correspondence between the sets $\{\mathcal{A}_{P}\}$ and $\{\mathcal{A}_{Q}\}$ appropriately.

A case of exceptional spin groups that are attached to the (at most $4g$) exceptional points of a surface $\Sigma$, (i.e. elements of exceptional leaves of the $\epsilon$-foliation) is given in ~\cite{KMB13}.

\end{document}